\documentclass[11pt]{article}
\usepackage[margin=1.2in]{geometry}
\usepackage{graphicx}
\usepackage{amssymb,amsmath,amsthm}
\usepackage{tikz}
\usetikzlibrary{arrows.meta}
\usepackage{authblk}

\newtheorem{theorem}{Theorem}
\newtheorem{lemma}{Lemma}
\newtheorem{definition}{Definition}

\newcommand{\ceil}[1]{\left\lceil#1\right\rceil}

\newcommand{\power}[2]{\left(#1\right)^{#2}}

\title{The Speed and Threshold of the Biased Hamilton Cycle Game}

\date{\today}

\author[1]{Noah Brustle}
\author[1]{Sarah Clusiau}
\author[1]{Vishnu V. Narayan}
\author[2]{Ndiam\'{e} Ndiaye}
\author[1]{Bruce Reed}
\author[3,4]{Ben Seamone\thanks{Emails:\\ \texttt{noah.brustle@mail.mcgill.ca, sarah.clusiau@mail.mcgill.ca,}\\ \texttt{vishnu.narayan@mail.mcgill.ca, ndiame.ndiaye@mail.mcgill.ca,}\\ \texttt{breed@cs.mcgill.ca, seamone@iro.umontreal.ca}}}
\affil[1]{School of Computer Science, McGill University, Montreal, Canada}
\affil[2]{Department of Mathematics and Statistics, McGill University, Montreal, Canada}
\affil[3]{Mathematics Department, Dawson College, Montreal, Canada}
\affil[4]{D\'{e}partement d'informatique et de recherche op\'{e}rationnelle, Universit\'{e} de Montr\'{e}al, Montreal, Canada}

\begin{document}
  \maketitle
\begin{abstract} 
We show that there is a constant C such that for any $b<\frac{n}{\ln{n}}-\frac{Cn}{(\ln{n})^{3/2}}$,
Maker wins the Maker-Breaker Hamilton cycle game in $n+\frac{Cn}{\sqrt{\ln{n}}}$ steps.
\end{abstract}

\section{Introduction}

The (Maker-Breaker) Hamilton cycle game on a graph $G$ is played by two players who alternately select edges from $G$, with Breaker choosing first.
Maker wins when she has chosen the edges of a Hamilton cycle. If she never does so, Breaker wins.   
We restrict our attention, and the definition of these games, to graphs $G$ that are cliques.

Chv\'{a}tal and Erd\H{o}s~\cite{CE78} proved that for sufficiently large $n$  and $G$ a clique on $n$ vertices, 
if both players play optimally then Maker wins the Hamilton cycle game, and can ensure she does so in her first $2n$ moves.  

Chv\'{a}tal and Erd\H{o}s~\cite{CE78} also introduced the biased version of this game where for some integer $b$, Breaker selects $b$ edges in each turn. They  focused on 
 the largest value of  $b$ for which Maker wins this game.  They showed that for any positive $\epsilon$, for $n$ sufficiently large, if $b>\frac{(1+\epsilon)n}{\ln{n}}$  then 
Breaker can ensure that he has selected all the edges incident to one of the vertices. Hence Breaker wins  the  biased Hamilton cycle  game
for such values of $b$. Krivelevich~\cite{Kri11} obtained an essentially matching lower bound, showing that for   $b<\frac{(1-\epsilon)n}{\ln{n}}$  
Maker  wins the biased Hamilton cycle  game\footnote{\cite{Kri11} states a slightly stronger result but a careful reading of the proof 
shows that the  result stated here is actually what is proven. In particular the use of Lemma 3  stated therein means the approach given there can do no better.}. 

For  Maker to win the (unbiased) Hamilton cycle  game  she must make at  least $n$ moves so as to obtain the edges of the cycle.  Further Breaker can prevent Maker from just choosing 
the edges of a Hamilton cycle  by stealing the last edge. So Maker cannot ensure she wins in fewer than $n+1$ moves. Hefetz and Stich~\cite{HS09} proved that this bound is tight; i.e., that she can always win in this many moves.  

In this paper we focus on the number of moves Maker needs to win the biased Hamilton cycle  game.  We are particularily interested in 
the values of $b$ for which Maker wins in $n+o(n)$ moves. This question has been previously studied by 
Mikalacki and Stojakovic in \cite{MS17}. They  showed  if $b=o\left(\sqrt{\frac{n}{(\ln{n})^5}}\right)$ then Breaker can win in $(1+o(1))n$ moves. 
They asked whether this remained true for larger values of $b$. We answer this question in the affirmative by showing:

\begin{theorem}
\label{themaintheorem}
There is a constant $C$ such that  for  $n$ sufficiently large  if $b \le \frac{n}{\ln{n}}-\frac{Cn}{(\ln{n})^{3/2}}$ then Maker can win the biased Hamilton cycle  game in $n+\frac{Cn}{(\ln{n})^{1/2}}$ steps. 
\end{theorem}

We note that this improves the upper bound from \cite{Kri11} on the bias that ensures that Maker can win this game. 

In Section \ref{sec:sketch}, we discuss the 3-stage strategy of Krivelevich, and provide a broad overview of our new 2-stage approach for Maker.  Section \ref{sec:strategy} is devoted to the formalization of Maker's strategy.  The proof of Theorem \ref{themaintheorem} is split over the final two sections of the paper.  Section \ref{sec:time} contains an analysis of the length of the game, and the proof that the strategy may be successfully implemented is contained in Section \ref{sec:thedetails}.

We close this section with some basic definitions. At any point in the game, for a vertex $v$, we use $d_B(v)$ to denote the degree of $v$ in the graph $B$ chosen by Breaker, and $d_M(v)$ to denote its degree in the graph $M$ chosen by Maker. We let $E_B$ and $E_M$ be the edges picked by Breaker and Maker respectively.

\section{A Proof Sketch}\label{sec:sketch} 
The approach we use to prove Theorem \ref{themaintheorem} builds on the technique adopted by Krivelevich in \cite{Kri11}, who showed  Maker can win in  $14n$ turns
if the bias is at most $\frac{(1-\epsilon)n}{\ln{n}}$. To do so he showed that there is a random strategy for Maker which yields a positive probability that 
she wins within $14n$ turns .  Since the game  (curtailed after $14n$ turns) cannot be a draw, with optimal play the same  player always wins,  so this implies his result. 

The random strategy Krivelevich suggested for Maker  is a 3-stage approach, whose second two phases we collapse to one. 
In the first phase Maker, with positive probability, chooses a spanning subgraph  $F$  with  good expansion properties that has at most $12n$ edges.  Specifically for some constant $\alpha=\alpha_{\epsilon}$, there is a postive probability   that for any set $X$ of at most $\alpha n$ vertices, the number of vertices in $V-X$ with a neighbour in $X$ has size at least $2|X|$. We note that this implies every component of $F$  has at least $\alpha n$ vertices. 

Krivelevich showed that  the expansion properties ensured  that for any longest path $P$ of  any supergraph $F'$  of $F$  chosen by Maker, if $v$ is an endpoint of P, then there are at least 
$\alpha n$ vertices that were endpoints of a path $P'$ on $V(P)$ whose other endpoint was $v$. Applying this result to each of these $\alpha n$ vertices we obtain that 
there are at least $(\alpha n)^2$ pairs of vertices that can be the pair of endpoints  of a path on $V(P)$. Because there will only be $14n$ turns, Breaker 
chooses  fewer than  $\frac{14n^2}{\ln{n}}$ edges.  It follows that one of these 
pairs are not the endpoints of an edge chosen by Breaker, and Maker can choose the edge between them. After she does so  there is a cycle  on $V(P)$ in her graph. 

If $V(P)$ is not a component of Maker's  graph,  then her graph contains   an edge $uv$ with $u$ not on P but $v$ on $P$  and hence a  path on $V(P) \cup \{u\}$ with one endpoint  $u$. 
So Maker has increased the length of the longest path in her chosen graph.

On the other hand,  If $V(P)$ is a component of her graph then it has at least $\alpha n$ vertices. If V(P) is the unique component then 
Maker  has chosen the edges of a  Hamilton cycle and is done. Otherwise there is another component $U$ of her graph  that by the expansion properties has at least $\alpha n$ vertices.
In the next turn, Maker chooses one of the  $(\alpha n)^2$  possible edges between $V(P)$  and $U$ thereby increasing the length of the longest path in her graph.

Our new strategy ensures that, with positive probability, the first phase ends with Maker having chosen the edges of a graph with desirable expansion properties on a set $S$ of $\Theta(\frac{n}{\sqrt{\ln{n}}})$ vertices, as well as the edges of a set of  $O(\frac{n}{\sqrt{\ln{n}}})$ paths that partition $V-S$.  The graph Maker has chosen at  this point has at most  $22|S|+|V-S|=n+o(n)$ edges.  In the second stage, Maker adapts
Krivelevich's approach, using the expansion property to extend the length of the longest path every two turns. However, she usually 
extends the longest path by many vertices adding all of the vertices on some  substantial subpath of one of the paths in the partition of V-S. This allows Maker to extend to 
a Hamilton path and then  a Hamilton cycle  in $o(n)$ turns. 

Maker's job is straightforward unless Breaker chooses a graph with high degree.
We define a vertex $v$ to be {\it troublesome} if $d_B(v)>\frac{n}{\sqrt{\ln{n}}}$.  We will show that, since there are fewer than $2n$ turns, $B$ has fewer than $\frac{2n^2}{\ln{n}}$ edges and hence there are fewer than  $\frac{2n}{\sqrt{\ln{n}}}$  troublesome vertices.  The fact that there are not ``too many'' troublesome vertices will allow Maker to adapt the straightforward strategy outlined in the previous paragraph to ensure that she can build the desired subgraph.

We finish the overview with some notation.  Whenever Maker chooses an edge she first chooses  a vertex  $v$ and then choose an edge  $vw$ incident to $v$.
We think of Maker building a directed graph where we direct this edge from $v$ to $w$. For a vertex $v$, we let $d'_M(v)$ be the outdegree 
of $v$ in Maker's (directed) graph. We let $d^-_M(v)$ be the number of edges from $v$ chosen before $v$ becomes troublesome and $d^+_M(v)$ 
be the number of edges chosen from $v$ after it becomes troublesome. So $d'_M(v)=d^-_M(v)+d^+_M(v)$ and if $v$ is not troublesome $d^+_M(v)=0$. The {\it danger} of a vertex $v$, denoted $\partial(v)$, is $d_B(v)-bd^+_M(v)$. 

\section{Maker's strategy}\label{sec:strategy}

Before beginning the game, Maker chooses an arbitrary set $S_0$ of $\lceil \frac{2000n}{\sqrt{\ln{n}}} \rceil $ vertices. 	
These vertices will all end up in $S$.  The family ${\cal F}_0$ of paths partitioning $V-S_0$ consists of $|V-S_0|$ singleton paths.  We will define 
for each $i>1$ a set $S_i$ of vertices and a partition ${\cal F}_i$ of $V-S_i$ into paths. 
In each phase, after Breaker's $(i+1)^{st}$ turn, every vertex of $V-S_i$ that becomes troublesome is added to $S_i$.
For each such vertex $v$ in turn, we delete the path P containing $v$ from ${\cal F}_i$ and add the components of 
$P-v$ to it. 

\subsection{Phase 1: Building skeleton paths}

In the  $i^{th}$ turn, Maker proceeds as follows, depending on which case applies.  Recall that, by the design of the strategy, all troublesome vertices currently lie in $S_{i-1}$.  We say that two paths $P$ and $P'$ of ${\cal F}_{i-1}$ are \textit{$uv$-available for Maker} if $u$ is an endpoint of $P$, $v$ is an endpoint of $P'$, $d_M(u),d_M(v) \le 1$, and $uv$ has not been claimed by Breaker.
\vskip0.2cm

\noindent\textbf{Case 1:} Every troublesome vertex has $d^+_M(v) \geq 10$.

\begin{enumerate}
 \item If there is a vertex $v$ in $S_{i-1}$ with $d'_M(v)<10$, then
    \begin{itemize}
        \item choose a uniform element $w^*$  of $\{w|w \in S_0, vw \notin B\}$ and add the edge $vw^*$ to $M$;
        \item set $S_i=S_{i-1}$ and ${\cal F}_i={\cal F}_{i-1}$.
    \end{itemize}

\item If every vertex of $S_{i-1}$ has $d'_M(v) \geq 10$, then
    \begin{enumerate}
        \item if there are two elements $P$ and $P'$ of ${\cal F}_{i-1}$ that are $uv$-available then
            \begin{itemize}
                \item add $uv$ to $M$;
                \item set $S_i=S_{i-1}$ and ${\cal F}_i={\cal F}_{i-1}-P-P'+P^*$ where $P^*$ is the concatenation of $P,uv$, and $P'$.
            \end{itemize}
        \item else end Phase 1, move to Phase 2, and set $i^*={i-1}$.
    \end{enumerate}
\end{enumerate}

We note that any vertex satisfying (1) cannot be troublesome.

\noindent\textbf{Case 2:} \textit{There is a troublesome vertex $v$ such that $d^+_M(v)<10$.}

\noindent Select $v$ to be a troublesome vertex with $d^+_M(v)<10$ having maximum danger (ties broken arbitrarily). Then
\begin{enumerate}
    \item[]\begin{itemize}
    \item choose any $w$ in $V-S_{i-1}$ such that $vw$ is not in $M$ or $B$;
    \item add $vw$ to $M$;
    \item set $S_i=S_{i-1} \cup \{w\}$ and ${\cal F}_i$ to be the set obtained from ${\cal F}_{i-1}$ by deleting the path $P$ containing $w$ and adding the components of $P-w$.
\end{itemize}
\end{enumerate}

\subsection{Phase 2: Adding boosters}

We note that Maker moves first in this stage.  Let $P_{i^*}$ denote a longest path in $M$ that contains at least one vertex from $S_{i^*}$.
She will build a sequence of paths $P_{i^*},P_{i^*+1},...,$ with the vertex set of $P_j$ contained in the 
vertex set of $P_{j+1}$.  At each stage, $P_i$ is chosen to be the longest path in $M$ that contains all vertices of $P_{i-1}$.

For $i>i^*$, Maker proceeds as follows on the $i^{th}$ turn depending on which case applies.\\

\noindent\textbf{Case 1:} Every troublesome vertex has $d^+_M(v) \geq 10$.

\begin{enumerate}
 \item If there is a vertex $v$ in $S_{i-1}$ with $d'_M(v)<10$, then
    \begin{itemize}
        \item choose a uniform element $w^*$  of $\{w|w \in S_0, vw \notin B\}$ and add the edge $vw^*$ to $M$;
        \item set $S_i=S_{i-1}$ and ${\cal F}_i={\cal F}_{i-1}$.
    \end{itemize}

\item If every vertex of $S_{i-1}$ has $d'_M(v) \geq 10$, then
    \begin{enumerate}
    \item if there is some $v$ that is an endpoint of a path of ${\cal F}_i$ that satisfies $d'_M(v)<10$, then
        \begin{itemize}
            \item choose a uniform element $w^*$ of $\{w|w \in S_0, vw \notin B\}$ and add the edge $vw^*$ to $M$;
            \item set $S_i=S_{i-1}$,${\cal F}_i={\cal F}_{i-1}$.
        \end{itemize}
    \item if every endpoint of a path in ${\cal F}_i$ satisfies $d'_M(v)\geq 10$, then
        \begin{enumerate}
            \item[(i)] if there is no cycle on $V(P_i)$ but there is a path on $V(P_i)$ with endpoints $u$ and $v$ in $S_0$ such that $uv$ not been chosen by Breaker, then
                \begin{itemize}
                    \item add $uv$ to $M$;
                    \item set $S_i = S_{i-1}$ and ${\cal F}_i={\cal F}_{i-1}$. 
                \end{itemize}
            \item[(ii)] else end Phase 2.
        \end{enumerate}
\end{enumerate}
\end{enumerate}

\noindent\textbf{Case 2:}\textit{ There is a troublesome vertex $v$ such that $d^+_M(v)<10$.}

\noindent Select $v$ to be a troublesome vertex with $d^+_M(v)<10$ having maximum danger (ties broken arbitrarily). Then
\begin{enumerate}
    \item[]\begin{itemize}
    \item choose any $w$ in $V-S_{i-1}$ such that $vw$ is not in $M$ or $B$;
    \item add $vw$ to $M$;
    \item set $S_i=S_{i-1} \cup \{w\}$ and ${\cal F}_i$ to be the set obtained from ${\cal F}_{i-1}$ by deleting the path $P$ containing $w$ and adding the components of $P-w$.
\end{itemize}
\end{enumerate}

\section{Analyzing the length of the game}\label{sec:time}

Assuming that Maker is able to build her desired subgraph as claimed, we now prove that the game terminates in $n+o(n)$ moves.

\begin{theorem}
If Maker wins by applying Phase 1 then Phase 2, then she does so in at most $n+\frac{52000n}{\sqrt{\ln{n}}}$ turns.
\end{theorem}

\begin{proof}
We derive our bound on the number of turns taken by bounding the sizes of $S_i$ and ${\cal F}_i$.

At the end of Phase 1, every vertex satisfies $d'_M(v) \le 20$ and so we have performed at most $20n$ turns.
This means that there are at most $\frac{20n^2}{\ln{n}}$ edges in $B$.  This implies that there are fewer than 
$\frac{5n}{\sqrt{\ln{n}}}$ paths of ${\cal F}_{i^*}$ both of whose endpoints have degree at most  1 in $M$, as otherwise we would have added an edge 
joining endpoints of two such paths rather than moving to Stage 2.  It also implies that there are at most $\frac{20n}{\sqrt{\ln{n}}}$ 
troublesome vertices, and hence $|S_{i^*}|  \le |S_0|+\frac{420n}{\sqrt{\ln{n}}} \le  \lceil \frac{2420n}{\sqrt{\ln{n}}} \rceil $.
We note further that for every vertex $v$ of $V-S_{i^*}$, we have $d_M(v) \le 2$.  Furthermore, if 
$v$ has a neighbour in $S_{i^*}$, then $v$ is the endpoint of a path of ${\cal F}_i$ and this neighbour is either troublesome, or is one of the 20 outneighbours of some troublesome vertex.
Furthermore, each vertex of $S_i$ is incident to at most two such edges. It follows that there are at most  $22|S_{i^*}|+|V-S_{i^*}|=n+O(\frac{n}{\sqrt{\ln{n}}})$  turns in the first stage and at most $\frac{420n}{\sqrt{\ln{n}}}$ paths in ${\cal F}_{i^*}$
that have an endpoint of degree 2 in $M$.  Thus, $|{\cal F}_{i^*}| \le \frac{425n}{\sqrt{\ln{n}}}$. 

Having obtained a bound of $n+o(n)$ on the number of turns, we can further strengthen our bounds on  $|S_{i^*}|$ and ${\cal F}_{i^*}$.
There are fewer than $\frac{2n}{\sqrt{\ln{n}}}$ troublesome vertices, so  $|S_{i^*}|  \le |S_0|+\frac{44n}{\sqrt{\ln{n}}} \le  \lceil \frac{2044n}{\sqrt{\ln{n}}} \rceil $.
Further, there are at most $\frac{2n}{\sqrt{\ln{n}}}$ paths in ${\cal F}_{i^*}$ that have an endpoint of degree at most 1, 
and so  $|{\cal F}_{i^*}| \le \frac{90n}{\sqrt{\ln{n}}}$. 

We now consider Phase 2.  Aside from those steps in which we make a cycle exist on $V(P_i)$ 
when no such cycle existed previously, there are at most $20n$ turns taken.  After having performed one such turn, we cannot perform another until $V(P_i)$ has grown bigger. Thus, we also perform at most $n$ such turns and in total at most $21n$ turns. As in the first phase this shows that there are $O(\frac{n}{\sqrt{\ln{n}}})$ troublesome vertices. 
Hence in total we perform at most $2n+o(n)$ turns in the two stages and there are at most $\frac{2n}{\sqrt{\ln{n}}}$ troublesome vertices. 
Thus, as in the first stage, we obtain that for every $i$, $|S_i|  \le |S_0|+\frac{44n}{\sqrt{\ln{n}}} \le  \lceil \frac{2044n}{\sqrt{\ln{n}}} \rceil $.
and  $|{\cal F}_i| \le \frac{90n}{\sqrt{\ln{n}}}$. 

We note that throughout the process, every interior vertex $u$  of a path $Q$ of ${\cal F}_i$ 
 has $d_M(u) \le 2$. So,  if $u$ is an interior vertex  of $P_i$ then  all of its neighbours on $Q$ 
 must be on $P_i$. On the other hand  if $u$ is an endpoint of $P_i$ then again all its neighbours must be on $P_i$, 
  or $P_i$ would not be maximal. 
 Thus, we see that throughout the process, if the interior of a path  $Q$ of ${\cal F}_i$ 
 intersects $P_i$ then  all of $Q$  is contained in $P_i$. Hence when we first add a vertex of the interior of $Q$  to $P_i$ we add all of the interior. 
 It follows that  the number of j between $i^*$ and $i$ for which $|P_{j-1}|<|P_j|$ is at most $|S_i|+3|{\cal F}_i|$.
 and hence we perform at most $|S_i|+3|{\cal F}_i|$ turns where we create a cycle on $V(P_i)$. 
 Hence,  throughout the two stages there are at most $n+23|S_i|+47|{\cal F}_i|$ turns.  The result follows.
\end{proof}

\section{Verifying Maker's Strategy}
\label{sec:thedetails}

It remains to show that Maker can carry out the strategies we have set out and that when the game finishes, she has created a Hamilton cycle.  

\begin{theorem}\label{t:choosing}
At each step of implementing Phase 1 then Phase 2, Maker is able to choose the edge she desires.
\end{theorem}

\begin{theorem}\label{t:winning}
At the end of implementing Phase 1 then Phase 2, $M$ is a Hamiltonian subgraph of $G$.
\end{theorem}

In order to prove these theorems, we require the following three lemmas, the details of which follow later.

\begin{lemma}
\label{l:secondstagelemma} 
If Maker implements Phase 1 and then Phase 2, then with probability $1-o(1)$, 
for every $i \ge i^*$ there are more than $\frac{20n^2}{\sqrt{\ln{n}}}$ pairs of vertices that can be the endpoints of a path on $V(P_i)$. 
\end{lemma}

\begin{lemma}
\label{l:lowdegreelemma} 
For sufficiently large C, if Maker implements Phase 1 and then Phase 2, then throughout the algorithm  every troublesome vertex with 
$d^+_M(v) <10$ satisfies $d_B(v) < n-\frac{3000n}{\sqrt{\ln{n}}}$.
\end{lemma}

\begin{lemma}
\label{l:connectivity} 
If Maker implements Phase 1 and then Phase 2, then with  probability $1-o(1)$  for every $i\ge i^*$, $S_i$ is connected.
\end{lemma}

Armed with these technical results, the proofs of our theorems easily follow.

\begin{proof}[Theorem \ref{t:choosing}]
We first show that, in each of Phase 1 and 2, any choice that needs to be made in Case 1 can be made.  Recall that  $|S_0|=\lceil \frac{2000n}{\sqrt{\ln{n}}} \rceil $. 
Thus if $v$ is a nontroublesome vertex with $d'_M(v) \le 10$, then the $w$ in $S_0$,  for which can choose $vw$ as an edge, form the 
overwhelming majority of $S_0$.  Lemma \ref{l:secondstagelemma} shows that Maker can always find $uv$-available paths in ${\cal F}_{i-1}$ when need be. 

To handle Case 2 in each of Phase 1 and Phase 2, we apply Lemma \ref{l:lowdegreelemma} (with $C$ taken sufficiently large) together with the observation that for every $i\le i^*$, $|S_i| \le  \lceil \frac{2044n}{\sqrt{\ln{n}}} \rceil $.
\end{proof}

\begin{proof}[Theorem \ref{t:winning}]
Since $P_i$ contains a vertex of $S_{i^*}$ and hence of $S_i$ for every $i \ge i^*$,  this Lemma \ref{l:connectivity} implies that the component of $M$ containing $P_i$ contains all of $S_i$.
Furthermore, in any turn in which Maker adds an edge to make $V(P_i)$ into a cycle, every endpoint of a path of $F_i$ has 10 neighbours in $S_i$ and hence  this implies that $M$ is connected. Thus Maker does indeed extend $P_i$ when she adds an edge that creates a cycle on $V(P_i)$ unless this is the last turn and she has created a Hamilton cycle.
\end{proof}

It remains to prove Lemmas \ref{l:secondstagelemma},  \ref{l:lowdegreelemma} and  \ref{l:connectivity}.  The techniques used closely mirror those used in a companion paper to this one on the biased perfect matching game \cite{B+20}.
 
 \subsection{The Proof of Lemma \ref{l:lowdegreelemma}}
 
 The proof that follows is a modification of a standard argument used by Krivelevich \cite{Kri11}. 
 
\begin{proof}Assume for the contrary that in some turn $f$, there is a  troublesome vertex $v$ with $d^+_M(v) <10$ and  $d_B(v) > n-\frac{3000n}{\sqrt{\ln{n}}}$.  We consider the last turn $s$ before this turn in which 
an edge was chosen from a non-troublesome vertex (such a turn must exist as at the start of the game there are no troublesome vertices) 
 and the suffix of $k$ turns after  it until turn $f$ . So,  in every turn between $s$ and $f$,  we 
chose an edge out of a troublesome vertex.
For each $i$ between $1$ and $k$,  we let $a_i$ be the vertex from which Maker chose an edge in the $i^{th}$ turn of this suffix 
and let $A_i$  be the set of $ \{a_j|  i \le j \le k\}$. Note that this is a set not a multi-set even though $a_j$ may equal $a_k$ for $i\neq k$. Note further  that  $A_k=\{v\}$.

Since, we chose an edge out of a troublesome vertex for each of these $k$ turns, we have $k \le \frac{44n}{\sqrt{\ln{n}}}$. 
For each $j$  between $0$ and $k$, and $x$ in $A_{j+1}$ we let         $d^j(x)= \deg_B^*(a_i) - |\{x| x \in A_{j+1}, a_ix \in E_B\}|-bd^+_M(a_i)$, after $j$ turns of the sequence.
        
 Before turn  $s$  there were no troublesome vertices and after it each vertex degree is increased by at most $b$.  So for each $x \in A_1$, we have:  $ d^0(x) \le  \frac{n}{\sqrt{\ln n}}+\frac{n}{\ln n}$. 

        We define the potential function $p(j)$, as follows:

$$p(j) := \frac{1}{|A_j|}\sum_{x \in A_j}d^{j-1}(x)$$

Note that, we are assuming that $p(k) = d^{k-1}(a_k) \ge d_B(a_k)-10b > n-\frac{3001n}{\sqrt{\ln{n}}}$. On the other hand  $$ p(1)  = \frac{1}{|A_1|}\sum_{x \in A_1}d^0 (x) \le \frac{1}{|A_o|} \times |A_0| \times \left(\frac{n}{\sqrt{\ln n}}+\frac{n}{\ln n}\right)=\frac{n}{\sqrt{\ln n}}+\frac{n}{\ln n}.$$

We derive a contradiction by bounding the increase in the potential function in each step. 

If $A_{j+1}=A_j$ then  since  the  $j^{th}$ turn increases $bd_M(a_j)$ by b and in any turn Breaker can  pick at most $b$ edges  with exactly one endpoint in $A_j$, we have:

$$ \sum_{x \in A_{j+1}}d^j(x) \le \sum_{x \in A_j}d^{j-1}(x)$$

and so $p(j+1) \le p(j)$. 

On the other hand, if $A_{j+1}=A_j-a_j$ then  $\sum_{x \in A_{j+1}}d^j(x)-\sum_{x \in A_{j+1}}d^{j-1}(x)$ is the sum of the number of 
edges Breaker chose in the $j^{th}$ turn with exactly one end in $A_{j+1}$ 
and  the number of edges he has chosen between $a_j$ and   $A_{j+1}$  by the end of the 
$j^{th}$ turn. The first of these is at most $b$ and the second is  is at most $|A_{j+1}|$. 

We obtain: 
$$(*) \sum_{x \in A_{j+1}}d^j(x) \le \sum_{x \in A_j+1}d^{j-1}(x) + b+|A_{j+1}|.$$

By  our definition of $d^j_i$ and choice of $a_i$ to maximize the danger ,
at the start of the $j^{th}$ turn,
$$\forall  x \in A_j,  d^{j-1}(x) \le  \partial(x) \le \partial(a_j) \le d^{j-1}(a_j)+ |A_{j+1}|$$
 
 Summing up and dividing by $|A_j|$, we obtain  $$d^{j-1}(a_j) \ge \frac{\sum_{x \in A_{j}}d^{j-1}(x)}{|A_j|} -{|A_{j+1}|}.$$
 
 Thus, 
 
 $$\sum_{x \in A_{j+1}}d^{j-1} (x) \le \frac{|A_{j+1}|}{|A_j|}\sum_{x \in A_j}d^j(x)+|A_{j+1}|.$$
 
Combining this with $(*)$ yields: 

$$\frac{\sum_{x \in A_{j+1}}d^{j}(x)}{|A_{j+1}|} \le \frac{\sum_{x \in A_j }d^{j-1}(x)}{|A_{j}|}+\frac{b}{|A_j|}+2.$$

Summing over the iterations where the potential can increase we obtain: 

$$p(k) \le p(1)+ 2|A_1| +b\sum_{r=1}^{|A_1|} \frac{1}{r}  \le p(1)+ 2k +b\sum_{r=1}^{k} \frac{1}{r} \le
 \frac{91n}{\sqrt{\ln{n}}}+\frac{n}{\ln{n}}+b(\ln{n})$$

Since $b<\frac{n}{\ln{n}}-\frac{Cn}{(\ln{n})^{3/2}}$ and  we are free to choose $C$ as large a constant as we like, 
the desired result follows.
\end{proof}

 \subsection{The Proofs of Lemmas \ref{l:secondstagelemma} and Lemma \ref{l:connectivity}}
 
Before proceding, we give a necessary definition.  
  \begin{definition}
  For $i \ge i^*$, we define $S'_i$ to be the union of $S_i$ and the endpoints of the 
  elements of ${\cal F}_i$.  For any subset $S$ of $S'_i$ we define
 $$N'(S)=\{v \in S_i \setminus S \,\mid\, \textup{$\exists$ $u \in S$ s.t.~the directed edge uv was chosen by Maker}\}.$$
  \end{definition}
  
In order to prove both lemmas, we first prove the following result, whose proof is again similar to one given by Krivelevich \cite{Kri11}.
  
  \begin{lemma}
  \label{l:expansion}
  If Maker implements Phase 1 then Phase 2, then with probability $1-o(1)$, 
  for every $i \ge i^*$ and $S \subseteq S'_i$, if $|S| \le \frac{|S'_i|}{2}$ then $N'(S)$ is nonempty
  while if  $|S| \le \frac{|S'_i|}{100}$,
  then $|N'(S)| >2|S|$. 
  \end{lemma}
  
\begin{proof}
  
  To calculate an upper bound on the  probability that there is some $S$ such that $N'(S)$ is so small, the conclusion here does not hold, we 
  simply sum an upper bound on  the probability over all choices of $i$, a partition of $S'_i$ into $S$, $A$ and $B$  with appropriate bounds on $|S|$ and $|A|$ of the probability that 
$N'(S)=|A|$.
We let $j=|S|$ and $k=|A|$ and note we can choose the partition by choosing $S \cup A$ and then  a partition of this set into $S$ and $A$.
Thus the number of such partitions is:
    \[\binom{|S_i'|}{j+k}\binom{j+k}{j}\le {\binom{|S'_i|}{j+k}}2^{j+k} \le \left(\frac{|S'_i|}{j+k}\right)^{j+k} (2e)^{j+k}\]
  For each such partition, Maker chooses at most two edges out of $v$ before $v$ is added to some $S_i$.
  From that point on, since $v$ is not troublesome,  Maker chooses eight edges from $v$ to $S_0$ and for each choice  there are most $\frac{n}{\sqrt{\ln{n}}}$ 
  edges from $v$ to $S_0$ that it cannot choose. Hence, the probability it chooses an edge from $v$ to $S+A$ is at most 
  $$\frac{|S|+|A|}{|S_0|-\frac{n}{\sqrt{\ln{n}}}} \le  \frac{j+k}{\frac{9|S'_i|}{10}}=\frac{j+k}{|S'|}\times\frac{10}{9}$$
  
  So the probability that we make choices so that $N'(S) \subseteq  A$  is at most 
  
   $$\left( \frac{j+k}{|S'_i|}\right)^{8j}\left(\frac{10}{9}\right)^{8j}.$$
   
  Thus, for a specific $j$ and $k$, the probability that there is some set $S$ of $j$ non-troublesome vertices of $S'_i$ such that 
  $N(S')=k$ is bounded above by: 
  
  $$\left( \frac{j+k}{|S'_i|}\right)^{7j-k}\left(\frac{10}{9}\right)^{8j}(2e)^{j+k}.$$
  
  Now, if $j<\frac{|S'_i|}{100}$ and $k<2j$ then this is less than 
  
  $$\left( \frac{j+k}{|S'_i|}\right)^{j} \left(\frac{1}{25}\right)^{3j}\left(\frac{200e}{81}\right)^{4j} \le \left(\frac{j+k}{|S'_i|}\right)^j\left(\frac{1}{5}\right)^{j}$$
  Assume that $x=(j+k)$ is fixed, then since $k\leq 2j$, we can conclude that $x\geq j\geq \frac{x}{3}$. Considering just one vertex of $S$ we know that $i=|S \cup N(S)|\geq 9$ so we must also have $j\geq 3$. So by summing over $j$ we get:
  \[\sum_{j=\max\left\{3,\ceil{\frac{x}{3}}\right\}}^{x} \power{\frac{x}{5|S'_i|}}{j}=\power{\frac{x}{5|S'_i|}}{\max\left\{3,\ceil{\frac{x}{3}}\right\}}\cdot O(1)\]
  
  Noting that $\power{\frac{x}{5|S_i'|}}{\max\left\{3,\ceil{\frac{x}{3}}\right\}}$
 is less than $\power{\frac{1}{|S_i'|}}{2}$ when $x\leq |S_i'|^{\frac{1}{3}}$ and less than $\power{\frac{1}{25}}{|S_i'|^{\frac{1}{3}}}$ otherwise since $j+k\leq 4j\leq \frac{|S_i'|}{25}$ by considering whether $x$ is greater or less than $|S_i'|^{\frac{1}{3}}$, by summing over $x$ we get:
  
  \[\sum_{x=9}^{\frac{|S_i'|}{25}} \power{\frac{x}{|S_i'|}}{\ceil{\frac{x}{3}}}\cdot O(1)\leq  \left(|S_i'|^{\frac{1}{3}}\cdot \power{\frac{1}{|S_i'|}}{3}+\frac{|S_i'|}{25}\cdot \power{\frac{1}{25}}{|S_i'|^{\frac{1}{3}}}\right)\cdot O(1)=O\left(\frac{1}{|S_i'|^{\frac{3}{2}}}\right)\]
  Finally, by noting that $|S_i'|\geq |S_0|=\Omega\left(\frac{n}{\sqrt{\ln{n}}}\right)$, we know that summing this over $i$ gives us $o(1)$ since $i=O(n)$.
  
  We consider next $\frac{|S'_i|}{100} \le  j \le \frac{|S'_i|}{2}$ and $k=0$. 
   As noted above, for any such choice of  $j$ and $k$, the probability that there is some set $S$ of $j$ nontroublesome vertices of $S'_i$ such that 
  $N(S')=k$ is bounded above by: 
  
   \begin{align*}
    &\power{\frac{j+k}{|S_i'|}}{7j-k} \power{\frac{10}{9}}{8j} (2e)^{j+k}\\
    = &\power{\power{\frac{j}{|S_i'|}}{7} \power{\frac{10}{9}}{8} 2e}{j}\\
   \le &\power{\frac{20e}{9}\cdot \power{\frac{5}{9}}{7}}{j}\\
   \le &\power{\frac{1}{100}}{j}
   \end{align*}
      Given our lower bound on $j$ and the sum of the first terms of a geometric sequence, the sum of this probability over all choices of $j$ is $\power{\frac{1}{100}}{\frac{n}{100\sqrt{\log n}+100}}\cdot O(1)=o\left(\frac{1}{n^2}\right)$ and summing over $i$,
   we again obtain a probability which is $o(1)$.
   \end{proof}

We may now prove Lemmas \ref{l:secondstagelemma} and \ref{l:connectivity}.

\begin{proof}[Lemma \ref{l:connectivity}]
For any component $S$ of $S_i$, $N'(S)$ is empty, thus Lemma \ref{l:expansion} 
  implies  that almost surely any such component has more than half the vertices, and hence $S_i$ is connected.
  \end{proof}
  
\begin{proof}[Lemma \ref{l:secondstagelemma}]
  We need to show that  if   for every $S \subseteq S'_i$ with $|S| \le \frac{|S'_i|}{100}$,
  we have $|N'(S)| >2|S|$ then there are more than $\frac{20n^2}{\sqrt{\ln{n}}}$ pairs of vertices that can be the endpoints of a path on $V(P_i)$. 

  We claim first that for every $i$ there is a path on $V(P_i)$ both of whose endpoints are in $S'_i$. 
  To see this consider a path $P'$  on $V(P_i)$ with the maximum number of endpoints in $S'_i$. If $P'$  contains an endpoint $w$ not in $S'_i$ 
  then  $w$ is on the interior of some element of ${\cal F}_i$. Hence, as we have  already seen,  it has degree 2 and both  of its 
  neighbours  are  on $P'$. If   $w$ sees the other endpoint of $P'$ then $P'$ extends to a cycle $C$  on $V(P')$. Now $P'$ contains a vertex of $S_i$.
  Further, Maker has chosen no  edge between this vertex and a vertex not in $S'_i$; Hence  $C$ contains  an  edge between two vertices in $S'_i$ 
  and deleting this edge yields a path on $V(P_i)$ with both endpoints in $S'_i$. Otherwise, $w$ is joined to an internal vertex $x$ of $P$ by an edge 
  that is not on $P$. This means that $x$ has degree three in $M$ and is the endpoint of the element of ${\cal F}_i$ containing 
  $w$. Now letting $xy$ be the first edge on the subpath of $P$ from $x$ to $w$, there is a path $P'$ with edge set $E(P)-xy+wx$ that has  $y$ instead 
  of $x$ as an endpoint  and shares  its other endpoint  with $P$. Now, $y$ is in $S'_i$, as $x$ has only one neighbour outside of $S'_i$. This contradicts our choice of $P$, and completes the proof of our claim. 
  
  We now consider a fixed path  $P$ on $V(P_i)$ with endpoints  $u$ and $v$ which are both in $S'_i$.
  We construct a family of paths ${\cal F}$  starting from $P$  all of which have $v$ as an endpoint, using {\it limited rotations}.
  Given a path $Q$ in the family with endpoints $v$ and $w$, we say $Q'$ is a {\it rotation} of  $Q$ if for some edge $wx$
  with $x \in V(Q)-v$, $E(Q')=E(Q)+wx-xy$ for some edge $xy$ of $Q$. We say the rotation is {\it  limited} if $x$ belongs to 
  $S_i$ and hence $y$ belongs to $S'_i$. We choose the   family ${\cal F}$  of paths that can be reached from $P$  via a sequence of limited rotations. 
  We let $S$ be the set of vertices other than $v$ that are endpoints of some path in the family. 
  
  We claim that $N'(S)$ is contained in the union of the neighbours on $P$ of the vertices of $S$. This implies that $|N'(S)| <2|S|$ 
  and hence, applying Lemma \ref{l:expansion} that $|S| \ge \frac{|S'_i|}{100}$. Repeating this argument for every $u$ in $S$ we obtain, for each such $u$,  a set $S_u$ of at least $\frac{|S'_i|}{100}$
  vertices such that for every $w \in S_u$ there is a path on $V(P_i)$ with endpoints $u$ and $w$. Considering all the  $\frac{|S'_i|}{100}$
  families, we obtain  Lemma \ref{l:secondstagelemma}.
  
  To prove our claim, we need to show that for every $w$ in $S$ and edge $wx$ chosen by Maker with $x \in S_i$,   $x$ is in $S$ or adjacent along $P$ to a 
  vertex of $S$. To this end,  we consider a sequence of limited rotations which take us from $P$ to a path with endpoints $v$ and $w$.
  We stop the first time the rotation causes us  to delete an edge incident to $x$ or makes $w$ an endpoint. In the former case, 
  the edge  of $P$ deleted  has an endpoint in $S$ and we are done. In the latter case, we consider the edge $wx$. It is not on $P$ and $x$ has degree 2 in $P$  or we are done,
  So since we are in this case,  this edge is also not on the current path, and we can  perform a limited rotation 
  using it and hence $x$ is adjacent to an element of $S$ via the edge  of $P$ deleted in this rotation. 
\end{proof}
  
\bibliography{hamilton-cycle}{}
\bibliographystyle{plain}

  \end{document}